\documentclass{birkmult}
\def\deftimestamp#1<#2>{\newcommand\timestamp{\sf 20#2}}
\deftimestamp Time-stamp: <09/09/02 18:28:28 wachi@open>

\usepackage{amsmath, amsfonts, amssymb, amsthm}

\newtheorem{theorem}{Theorem}
\newtheorem{lemma}[theorem]{Lemma}
\newtheorem{proposition}[theorem]{Proposition}
\theoremstyle{definition}
\newtheorem{definition}[theorem]{Definition}
\theoremstyle{remark}
\newtheorem{remark}[theorem]{Remark}

\DeclareMathOperator{\In}{in}
\DeclareMathOperator{\gin}{gin}

\newcommand{\frakm}{\mathfrak{m}}
\newcommand{\PPii}{K[x_1,x_2]}
\newcommand{\PPiii}{K[x_1,x_2,x_3]}
\newcommand{\PPiv}{K[x_1,x_2,x_3,x_4]}
\newcommand{\PPn}{K[x_1,x_2,\ldots,x_n]}
\newcommand{\ZZ}{\mathbb{Z}}
\newcommand{\HH}{{\mathbf H}}

\begin{document}
%-------------------------------------------------------------------------
% editorial commands: to be inserted by the editorial office
%
%\firstpage{1}
%\volume{228}
%\Copyrightyear{2004}
%\DOI{003-0001}
%
%
%\seriesextra{Just an add-on}
%\seriesextraline{This is the Concrete Title of this Book\br H.E. R and S.T.C. W, Eds.}
%
% for journals:
%
%\firstpage{1}
%\issuenumber{1}
%\Volumeandyear{1 (2004)}
%\Copyrightyear{2004}
%\DOI{003-xxxx-y}
%\Signet
%\commby{inhouse}
%\submitted{March 14, 2003}
%\received{March 16, 2000}
%\revised{June 1, 2000}
%\accepted{July 22, 2000}
%
%
%
%---------------------------------------------------------------------------
%Insert here the title, affiliations and abstract:
%
\title[Generic initial ideals]
{Generic initial ideals of some monomial complete intersections
  in four variables}
%----------Author 1
\author{Tadahito Harima}
\address{
  Faculty of Education, \br
  Ehime University, \br
  Ehime 790-8577, JAPAN}
% \email{}
% \thanks{The first author was supported by Grant-in-Aid
  %   for Scientific Research (C) (20540035)}
%----------Author 2
\author{Sho Sakaki}
\address{
  Department of Mathematics, \br
  Hokkaido University of Education, \br
  Kushiro 085-8580, JAPAN}
%----------Author 1
\author{Akihito Wachi}
\address{
  Division of Comprehensive Education, \br
  Hokkaido Institute of Technology, \br
  Sapporo 006-8585, JAPAN} 
% \email{wachi@hit.ac.jp}
%----------classification, keywords, date
\subjclass{
  Primary
  13A02; % Graded rings [See also 16W50]
  Secondary 
  13C40, % Linkage, complete intersections and determinantal ideals
  13F20, % Polynomial rings and ideals; rings of integer-valued polynomials
  13D40. % Hilbert-Samuel and Hilbert-Kunz functions; Poincare' series
}

\keywords{
  generic initial ideal,
  monomial complete intersection,
  strong Lefschetz property,
  almost revlex ideal
}

\date{\small compiled: {\sf\today}, saved: \timestamp}
%----------additions
%\dedicatory{To my boss}
%%% ----------------------------------------------------------------------

\begin{abstract}
Let $R = \PPiv$ be the polynomial ring 
over a field of characteristic zero.
For the ideal
$(x_1^a, x_2^b, x_3^c, x_4^d) \subset R$,
where at least one of $a$, $b$, $c$ and $d$ is equal to two,
we prove that its generic initial ideal 
with respect to the reverse lexicographic order
is the almost revlex ideal corresponding to the same Hilbert function.
\end{abstract}

%%% ----------------------------------------------------------------------
\maketitle
%%% ----------------------------------------------------------------------
%\tableofcontents

%=====================================================================
\section{Introduction}

Generic initial ideals
play an important role in commutative ring theory.
But it is very difficult to determine them,
and there are only a few results which determine 
generic initial ideals.
Even in the case of monomial complete intersections,
their generic initial ideals are not determined in general.
Our result is a starting point of this problem.
In the polynomial rings with one or two variables,
generic initial ideals are trivially determined,
since Borel-fixed ideals are unique for Hilbert functions.
In the case of three variables,
due to the result of Ahn-Cho-Park \cite{MR2371960} 
or Cimpoea{\c{s}} \cite{MR2292574},
the generic initial ideals of 
Artinian monomial complete intersections are determined.
In this note we focus on the case of four variables.
For the monomial complete intersections
$(x_1^a, x_2^b, x_3^c, x_4^d)$,
where at least one of $a$, $b$, $c$ and $d$ is equal to two,
we prove that their generic initial ideals 
are the almost revlex ideals
(Theorem~\ref{thm:gin-of-4var-ci}).

Throughout this note,
$K$ denotes a field of characteristic zero,
and $R$ the polynomial ring over $K$.
The only term order on $R$ used in this note is
the reverse lexicographic order 
with $x_1 > x_2 > \cdots$,
and $\gin(I)$ (resp. $\In(I)$) denotes the generic initial ideal
(resp. initial ideal) with respect to 
the reverse lexicographic order.

%=====================================================================
\section{The $k$-strong Lefschetz property}
\label{sec:k-slp}

In this section
we review the definition of $k$-strong Lefschetz property
and results needed for our main theorem.

%---------------------------------------------------------------------
\begin{definition}[the SLP and the $k$-SLP]
\label{defn:slp} \label{defn:k-slp}
Let $A$ be a graded Artinian algebra over a field $K$,
and $A = \bigoplus_{i=0}^c A_i$ its decomposition
into graded components.

(1) The algebra $A$ is said to have the
{\em strong Lefschetz property} ({\em SLP} for short),
if there exists an element $\ell \in A_1$ such that
the multiplication map
$\times \ell^s: A_i \to A_{i+s}$ ($f \mapsto \ell^s f$)
is full-rank for every $i \ge 0$ and $s>0$.
In this case, $\ell$ is called a {\em Lefschetz element},
and  we also say that $(A,\ell)$ has the SLP.

(2)
Let $k$ be a positive integer.
The algebra $A$ is said to have the
{\em $k$-strong Lefschetz property} ({\em $k$-SLP} for short),
if there exist linear elements $g_1, g_2, \ldots, g_k \in A_1$
satisfying the following two conditions.
\begin{itemize}
  \item[(i)]
  $(A, g_1)$ has the SLP,
  \item[(ii)]
  $(A/(g_1, \ldots, g_{i-1}), g_i)$ has the SLP
  for all $i = 2, 3, \ldots, k$.
\end{itemize}
In this case, we say that 
$(A, g_1, \ldots, g_k)$ has the $k$-SLP.
In other words,
$A$ is said to have the $k$-SLP,
if $A$ has the SLP with a Lefschetz element $g_1$, 
and $A/(g_1)$ has the $(k-1)$-SLP.

Note that the $1$-SLP is nothing but the SLP,
and that if a graded algebra has the $k$-SLP, then
it has the $(k-1)$-SLP.
Note also that the $n$-SLP is equivalent to the $(n-2)$-SLP
for the quotient rings $\PPn/I$,
since all graded $K$-algebras $K[x_1]/J$ and $\PPii/J$ have the SLP
\cite[Theorem~4.4]{MR1970804}. 
\end{definition}
%---------------------------------------------------------------------
\begin{definition}[almost revlex ideals]
\label{defn:almost-revlex}
A monomial ideal $I$ is called an {\em almost revlex ideal},
if the following condition holds:
for each monomial $u$ in the minimal generating set of $I$,
every monomial $v$ with $\deg v = \deg u$ 
and $v >_{\text{revlex}} u$ belongs to $I$.

It is clear that if two almost revlex ideals have 
the same Hilbert function, then they are equal.
In addition, it is easy to see that
almost revlex ideals are Borel-fixed
\cite[Remark~11]{harima-wachi-comm-alg(was07072247)}.
\end{definition}
%---------------------------------------------------------------------
We write Hilbert functions of graded algebras 
as $h$-vectors $h = (h_0, h_1, \ldots, h_c)$.
A Hilbert function $h$ is said to be {\em unimodal},
if there exist an integer $a$ such that
$h_0 \le h_1 \le \cdots \le h_a \ge h_{a+1} \ge \cdots \ge h_c$.
A Hilbert function $h = (h_0, h_1, \ldots, h_c)$ ($h_c \ne 0$)
is said to be {\em symmetric},
if $h_i = h_{c-i}$ for every $i \ge 0$.
The {\em difference} $\Delta h$ of $h$ is defined by
\begin{align*}
  (\Delta h)_i &= \max\{h_{i} - h_{i-1}, 0\}
  \qquad (i = 0, 1, 2, \ldots),
\end{align*}
where $h_{-1}$ is defined as zero.
We define the $k$th difference $\Delta^k h$ 
by applying $\Delta$ to $h$ $k$-times. 
The following is a direct consequence of 
\cite[Corollary~27]{harima-wachi-comm-alg(was07072247)}
%---------------------------------------------------------------------
\begin{proposition}
\label{prop:gin-for-4-slp}
Let $I \subset \PPiv$ be a graded Artinian ideal 
whose quotient ring has the 2-SLP.
Suppose that the Hilbert function $h$ of $\PPiv/I$ is symmetric.
Then the generic initial ideal $\gin(I)$
is the unique almost revlex ideal for the Hilbert function $h$.
\qed
\end{proposition}
%---------------------------------------------------------------------

We conclude this section by an analogue of 
Wiebe's result \cite[Proposition~2.9]{MR2111103}.
%---------------------------------------------------------------------
\begin{proposition}
\label{prop:I-is-kSLP-if-in(I)-is-kSLP}
Let $I$ be a graded Artinian ideal of $R=\PPn$,
and let $1\leq k\leq n$. 
If $R/\In(I)$ has the $k$-SLP, then $R/I$ has the $k$-SLP. 
\end{proposition}
%---------------------------------------------------------------------
\begin{proof}
Let $\HH_A(t)$ denotes the Hilbert function of a graded algebra $A$.
From the proof of \cite[Proposition~2.9]{MR2111103}, 
we have  
\begin{align}
  \label{eq:conca2}
  \HH_{R/(I+(g_1,\ldots,g_{i-1},g_i^s))}(t) \leq 
  \HH_{R/(\In(I)+(g_1,\ldots,g_{i-1},g_i^s))}(t)
\end{align}
for generic linear forms $g_1,\ldots,g_i$, 
$s \geq 1$ and all $t\geq 0$. 
In order to prove our claim, 
it is enough to show that 
the Hilbert function of $R/(I+(g_1,\ldots,g_{i-1},g_i^s))$ 
coincides with that of $R/(\In(I)+(g_1,\ldots,g_{i-1},g_i^s))$ 
for every $i=1,2,\ldots,k$ 
under the assumption that the Hilbert function of
 $R/(I+(g_1,\ldots,g_{i-1}))$ 
is equal to that of $R/(\In(I)+(g_1,\ldots,g_{i-1}))$. 
Set $h = \HH_{R/(I+(g_1,\ldots,g_{i-1}))}$.
By our assumption, 
$R/(\In(I)+(g_1,\ldots,g_{i-1}))$ has the SLP.  
Hence, it follows that 
the Hilbert function of $R/(\In(I)+(g_1,\ldots,g_{i-1},g_i^s))$ 
is equal to the sequence $(b_t)_{t \ge 0}$:
\begin{align*}
  b_t = \max\{ h_t - h_{t-s}, 0 \},
\end{align*}
where $h_t = 0$ for $t < 0$.
Furthermore one can easily check that 
\begin{align*}
  b_t \leq \HH_{R/(I+(g_1,\ldots,g_{i-1},g_i^s))}(t)
\end{align*}
for all $t \ge 0$.
Hence it follows from (\ref{eq:conca2}) that 
\begin{align*}
  \HH_{R/(I+(g_1,\ldots,g_{i-1},g_i^s))}(t) 
  = \HH_{R/(\In(I)+(g_1,\ldots,g_{i-1},g_i^s))}(t)
\end{align*}
for all $t \geq 0$. 
\end{proof}
%---------------------------------------------------------------------

%=====================================================================
\section{Main theorem}
\label{sec:main-theorem}

In this section we prove the main theorem 
(Theorem~\ref{thm:gin-of-4var-ci}).
For two graded Artinian $K$-algebras $A$ and $B$ having the SLP,
$A \otimes_K B$ also has the SLP
if their Hilbert functions are symmetric \cite{MR951211}.
But this is not the case unless both Hilbert functions are symmetric
(see \cite[Example~5]{MR2033004}, e.g.).
The following lemma gives a necessary and sufficient condition
for $A \otimes_K K[y]/(y^2)$ to have the SLP, 
when the Hilbert function of $A$ is not necessarily symmetric.

%---------------------------------------------------------------------
\begin{lemma}
\label{lem:(51.1)}
Let $A$ be a graded Artinian $K$-algebra
having the SLP,
and $B = K[y]/(y^2)$.
The tensor product $A \otimes_K B$ has the SLP,
if and only if the Hilbert function $h = h_A$ of $A$
satisfies the following two conditions:
\begin{quote}
\textsf{(C1)}
For any $h_i < h_{i+1}$,
there exist at most one $j$ such that $h_i < h_j < h_{i+1}$.
\par\noindent
\textsf{(C2)}
For any $h_i > h_{i+1}$,
there exist at most one $j$ such that $h_i > h_j > h_{i+1}$.
\end{quote}
\end{lemma}
%---------------------------------------------------------------------
\begin{proof}
First we recall that, 
given a graded Artinian $K$-algebra $A$ 
and a linear form $g\in A$, 
we may consider $A$ a $K[x]$-module via $x\ast a=ga$. 
Then, 
by the structure theorem of finitely generated module over PID, 
$A$ decomposes uniquely as direct sum of 
$K[x]/(x^{a_i})[-b_i]$ 
where the shift $[-b_i]$ indicates that 
the generators sits in degree $b_i$. 
So the Hilbert function of $A$ is 
$\sum \lambda^{b_i}(1-\lambda^{a_i})/(1-\lambda)$. 
Which $b_i$ actually occur is given by the Hilbert function of $A/(g)$ 
which is $\sum \lambda^{b_i}$. 
Call the pairs $(b_i, a_i)$ the basic invariants of $(A,g)$ 
and we can easily prove that 
$y$ is a SL element for $A$ if and only if 
the basic invariants have the following properties: 
\begin{gather}
\label{eq:cond-slp-by-basic-inv}
\text{if $(b_i, a_i)$ and $(b_j, a_j)$ are
  basic invariants with $b_i<b_j$ then $b_i+a_i\geq b_j+a_j$.}
\end{gather}

Next we consider the tensor product 
$C=K[x]/(x^a)\otimes_K B \simeq K[x,y]/(x^a, y^2)$ 
as a $K[z]$-module, 
where the action of $K[z]$ induced on $C$ 
is given by $z\ast f(x,y)=(x+y)f(x,y)$. 
Then, thanks to \cite[Proposition 8]{MR2033004}, 
the tensor product $D=K[x]/(x^{a_i})[-b_i]\otimes_K B$ 
decomposes into two modules, 
$D=K[z]/(z^{a_i+1})[-b_i]\oplus K[z]/(z^{a_i-1})[-(b_i+1)]$. 
If $a_i=1$, 
then the second component does not appear in $D$. 

Thus $A\otimes_K B$ has the SLP, 
if and only if 
every combination of two basic invariants of $A\otimes_K B$
satisfies Condition~(\ref{eq:cond-slp-by-basic-inv}). 

Furthermore we consider the following conditions: 
\begin{quote}
\textsf{(C1)$'$}\;
If $b_i=b_j$ then $\mid a_i-a_j\mid\leq 1$. 
\par\noindent
\textsf{(C2)$'$}\;
If $b_i+a_i=b_j+a_j$ then $\mid b_i-b_j\mid\leq 1$. 
\end{quote}
Note that these conditions are equivalent to our conditions, 
that is, \textsf{(C1)} is equivalent to \textsf{(C1)$'$},
and \textsf{(C2)} is equivalent to \textsf{(C2)$'$}. 

Suppose that $A\otimes_K B$ has the SLP. 
Hence the basic invariants $\{(b_i, a_i+1), (b_i+1, a_i-1)\}$ 
of $A\otimes_K B$ satisfy
Condition~(\ref{eq:cond-slp-by-basic-inv}).
First assume that there exist $i$ and $j$ 
such that $b_i=b_j$ and $a_i\geq a_j+2$. 
Then $b_i+1 > b_j$ and $(b_i+1)+(a_i-1)>b_j+(a_j+1)$. 
This means that 
two basic invariants $(b_i+1,a_i-1)$ and $(b_j, a_j+1)$ do not satisfy
Condition~(\ref{eq:cond-slp-by-basic-inv}).
Next assume that there exist $i$ and $j$ 
such that $b_i+a_i=b_j+a_j$ and $b_i+2\leq b_j$. 
Then $b_i+1 < b_j$ and $(b_i+1)+(a_i-1)<b_j+(a_j+1)$. 
This also means that 
two basic invariants $(b_i+1,a_i-1)$ and $(b_j, a_j+1)$ do not satisfy
Condition~(\ref{eq:cond-slp-by-basic-inv}). 
Thus the Hilbert function of $A$ satisfies 
Conditions \textsf{(C1)} and \textsf{(C2)}. 

Conversely suppose that the basic invariants
$(b_i, a_i)$ satisfy Conditions \textsf{(C1)$'$} and \textsf{(C2)$'$}. 
We can check that the basic invariants
$\{(b_i, a_i+1), (b_i+1, a_i-1)\}$ of $A \otimes_K B$ satisfy
Condition~(\ref{eq:cond-slp-by-basic-inv}) as follows.
Take two basic invariants $(b_i, a_i+1)$ and $(b_j+1, a_j-1)$
for example.
If $b_i < b_j+1$, then there are two possibilities.
(i) When $b_i = b_j$,
we have $|a_i-a_j| \le 1$ from \textsf{(C1)$'$}.
Hence
$b_i+(a_i+1) = (b_j+1)+(a_j-1) + (a_i-a_j+1) \ge (b_j+1)+(a_j-1)$.
(ii) When $b_i < b_j$,
we have $b_i+a_i \ge b_j+a_j$ 
from Condition~(\ref{eq:cond-slp-by-basic-inv}) for $A$.
Hence $b_i+(a_i+1) > (b_i+1)+(a_j-1)$,
and Condition~(\ref{eq:cond-slp-by-basic-inv}) 
for $A \otimes_K B$ is satisfied.
If $b_i > b_j+1$, then we have $b_i+a_i \ne b_j+a_j$
from the contraposition of \textsf{(C2)$'$}
and $b_i+a_i \le b_j+a_j$ 
from Condition~(\ref{eq:cond-slp-by-basic-inv}) for $A$.
Hence $b_i+(a_i+1) \le (b_j+1)+(a_j-1)$.
Thus Condition~(\ref{eq:cond-slp-by-basic-inv})
for $A \otimes_K B$ is satisfied.
Calculations are similar for other choices 
$(b_i,a_i+1)$ and $(b_j,a_j+1)$, or
$(b_i+1,a_i-1)$ and $(b_j+1,a_j-1)$
of basic invariants,
and thus $A \otimes_K B$ has the SLP.
\end{proof}
%---------------------------------------------------------------------
The following two lemmas give sufficient conditions for 
the tensor product $A \otimes_K K[y]/(y^2)$
to have the $k$-SLP.
In Lemma~\ref{lem:(53.1)},
$A$ is a quotient ring by a Borel-fixed ideal
having the $k$-SLP.
In Lemma~\ref{lem:(53.2)},
$A$ is any graded algebra having the $k$-SLP.
%---------------------------------------------------------------------
\begin{lemma}
\label{lem:(53.1)}
Let $A = R/I$ be a graded Artinian $K$-algebra having the $k$-SLP,
where $R = \PPn$, and $I$ is a Borel-fixed ideal of $R$.
Let $h_A$ be the Hilbert function of $A$.
Let $B = K[y]/(y^2)$.

Suppose that every $s$-th difference $\Delta^s h_A$
($0 \le s \le k-1$) satisfies 
Conditions \textsf{(C1)} and \textsf{(C2)} of
Lemma~\ref{lem:(51.1)}.
Suppose also that every $s$-th difference $\Delta^s h_A$
($0 \le s \le k-2$) satisfies the following condition:
\begin{quote}
\textsf{(C3)}
For $h = (h_0, h_1, \ldots, h_c)$,
there are two or more $i$ such that $h_i = \max_j\{h_j\}$,
or there is only one $i$ such that $h_i = \max_j\{h_j\}$,
and $h_{i-1} \ge h_{i+1}$.
\end{quote}
Then the tensor product $A \otimes_K B$ has the $k$-SLP.
\end{lemma}
%---------------------------------------------------------------------
\begin{proof}
We prove the lemma by induction on $k$.
For $k=1$, the lemma follows from Lemma~\ref{lem:(51.1)}.

Let $k>1$, and assume that the lemma holds up to $k-1$.
Since $I$ is Borel-fixed, $x_n$ is a Lefschetz element of $A$
\cite[Lemma~2.7]{MR2111103}.
Hence by the assumption of induction, $A \otimes_K B$ has the SLP,
and $x_n + y \in A \otimes_K B \cong A[y]/(y^2)$ is 
a Lefschetz element \cite{MR951211}.
Thus it suffices to show that
$A \otimes_K B / (x_n+y)$ has the $(k-1)$-SLP.
We have
\begin{align*}
A \otimes_K B / (x_n + y)
&\simeq
K[x_1,x_2,\ldots,x_n,y] / (I) + (x_n+y) + (y^2)
\\&\simeq
\PPn / I+(x_n^2)
\\&\simeq
A/(x_n^2).
\end{align*}

[Case 1. There are two or more $i$ such that $h_i = \max_j\{h_j\}$]
In this case,
if a monomial $u \in R$ not divisible by $x_n$
is a {\em standard monomial} (i.e. a monomial not belonging to $I$),
then $u x_n$ is also a standard monomial,
since $x_n$ is a Lefschetz element.
Therefore we have an algebra isomorphism
$A/(x_n^2) \simeq A/(x_n) \otimes_K K[z]/(z^2)$,
since $A$ is a quotient of a Borel-fixed ideal.
Here $A/(x_n)$ has the $(k-1)$-SLP, 
and has the Hilbert function $\Delta h_A$.
Therefore it follows from the assumption of induction that
$A/(x_n^2) \simeq A/(x_n) \otimes_K K[z]/(z^2)$ has the $(k-1)$-SLP.

[Case 2.
There is only one $i$ such that $h_i = \max_j\{h_j\}$,
and $h_{i-1} \ge h_{i+1}$] 
In this case,
for a standard monomial $u \in R$ not divisible by $x_n$,
$u x_n$ is also standard if $\deg(u) < i$,
and is not standard if $\deg(u) = i$.
Therefore we have an algebra isomorphism
\begin{align*}
  A/(x_n^2) \simeq
  \left[ A/(x_n) \otimes_K K[z]/(z^2) \right] / \frakm^{i+1},
\end{align*}
where $\frakm$ is the graded maximal ideal of
$A/(x_n) \otimes_K K[z]/(z^2)$.
Namely $A/(x_n^2)$ is isomorphic to 
the algebra obtained by dropping the homogeneous component
of the socle degree of $A/(x_n) \otimes_K K[z]/(z^2)$.
In general, for an algebra having the $k$-SLP,
the algebra obtained by dropping the homogeneous component of
the socle degree again has the $k$-SLP.
Hence $A/(x_n^2)$ has the $(k-1)$-SLP.

In both cases we have proved that $A/(x_n^2)$ has the $(k-1)$-SLP,
and by induction we have proved the lemma.
\end{proof}
%---------------------------------------------------------------------
\begin{lemma}
\label{lem:(53.2)}
Let $A$ be a graded Artinian $K$-algebra having the $k$-SLP,
and $h_A$ its Hilbert function.
Let $B = K[y]/(y^2)$.
Suppose that every $s$-th difference $\Delta^s h_A$
($0 \le s \le k-1$) satisfies 
Conditions \textsf{(C1)} and \textsf{(C2)} of
Lemma~\ref{lem:(51.1)},
and suppose also that every $s$-th difference $\Delta^s h_A$
($0 \le s \le k-2$) satisfies
Condition \textsf{(C3)} of 
Lemma~\ref{lem:(53.1)}.
Then the tensor product $A \otimes_K B$ has the $k$-SLP.
\end{lemma}
%---------------------------------------------------------------------
\begin{proof}
Let $A = R/I$, where $R = \PPn$, 
and $I$ a graded ideal of $R$.
Let $g \in GL_n(K)$ be an element 
for which $\In(g I) = \gin(I)$,
and let $\tilde{g} \in GL_{n+1}(K)$ be the element
given by embedding $g$ into first $n$ dimensions.
Then we have
\begin{align}
\label{eq:lem(53.2)-proof}
\nonumber
R/\gin(I) \otimes_K B
& =
R/\In(g I) \otimes_K B
\\&\simeq \nonumber
R[y] / (\In(g I) + (y^2))
\\&= \nonumber
R[y] / \In(g I + (y^2))
\\&\simeq 
K[x_1,x_2,\ldots,x_n,y]/\In(\tilde{g}(I+(y^2))),
\end{align}
by use of \cite[Proposition~15.15]{MR1322960} for example.
Here $R/\gin(I)$ has the $k$-SLP,
since $R/I$ has the $k$-SLP 
if and only if $R/\gin(I)$ has the $k$-SLP
\cite[Proposition~18]{harima-wachi-comm-alg(was07072247)}.
It follows from Lemma~\ref{lem:(53.1)} that
$R/\gin(I) \otimes_K B$ has the $k$-SLP.
Hence $K[x_1,x_2,\ldots,x_n,y]/\tilde{g}(I+(y^2))$
has the $k$-SLP by Proposition~\ref{prop:I-is-kSLP-if-in(I)-is-kSLP}.
Thus $A \otimes_K B \simeq K[x_1,x_2,\ldots,x_n,y]/\tilde{g}(I+(y^2))$
has the $k$-SLP.
\end{proof}
%---------------------------------------------------------------------
We need the following property on Hilbert functions
of monomial complete intersections of three variables
for the proof of the main theorem.
%---------------------------------------------------------------------
\begin{lemma}
\label{lem:harima-sakaki}
Let $a$, $b$ and $c$ be positive integers.
Define the $h$-vector $h = (h_0, h_1, \ldots)$ by 
\begin{align}
\label{eq:lem:harima-sakaki}
\sum_{j=0}^{a+b+c-3} h_j t^j =
(1+t+\cdots+t^{a-1}) (1+t+\cdots+t^{b-1}) (1+t+\cdots+t^{c-1}).
\end{align}
Then its difference $h' = \Delta h = (h'_0, h'_1, \ldots, h'_j, \ldots)$ is
a piecewise linear function in $j$,
and the coefficient of $j$ is at least $-2$ in each linear piece.
\end{lemma}
%---------------------------------------------------------------------
\begin{proof}
The right-hand side of Equation~(\ref{eq:lem:harima-sakaki})
is equal to
\begin{align*}
\sum_{j=0}^{a+b+c-3}
\sum_{\substack{0 \le k \le a-1, \;\;
0 \le l \le b-1, \;\;
0 \le m \le c-1, \\
k+l+m = j}} t^{k+l+m},
\end{align*}
and hence $h_j$ is equal to the number of lattice points on the plane 
$P_j =
\{ v = (k,l,m) \in \ZZ^3 \;;\; k+l+m = j \text{~and~} k,l,m \ge 0 \}$
satisfying
$0 \le k \le a-1$,
$0 \le l \le b-1$ and
$0 \le m \le c-1$.
Thus we have
\begin{align*}
h_j &= 
\# P_j
\\ & \quad
- \# \{ v \in P_j \;;\; k \ge a \}
- \# \{ v \in P_j \;;\; l \ge b \}
- \# \{ v \in P_j \;;\; m \ge c \}
\\ & \quad
+ \# \{ v \in P_j \;;\; k \ge a, \; l \ge b \}
\\ & \quad
+ \# \{ v \in P_j \;;\; k \ge a, \; m \ge c \}
\\ & \quad
+ \# \{ v \in P_j \;;\; l \ge b, \; m \ge c \}
\\ & =
\binom{j+2}{2} 
- \binom{j-a+2}{2} 
- \binom{j-b+2}{2} 
- \binom{j-c+2}{2} 
\\ & \qquad
+ \binom{j-a-b+2}{2} 
+ \binom{j-b-c+2}{2}
+ \binom{j-a-c+2}{2} 
\end{align*}
for $0 \le j \le a+b+c-3$,
where the binomial coefficients $\binom{n}{m}$ are defined as zero
if $n<0$.
Note that the formula 
$\binom{n-1}{m-1} + \binom{n-1}{m} = \binom{n}{m}$
still holds unless $(n,m) = (0,0)$,
and that $\binom{n}{1} = \max\{0, n\}$.
Hence we have
\begin{align*}
& h_j - h_{j-1} = 
\\ & \quad
j+1 
- \max\{0,j-a+1\}
- \max\{0,j-b+1\}
- \max\{0,j-c+1\}
\\ & \quad
+ \max\{0,j-a-b+1\}
+ \max\{0,j-b-c+1\}
+ \max\{0,j-a-c+1\}.
\end{align*}
Therefore $h'_j = \max\{0, h_j - h_{j-1}\}$
is a piecewise linear function in $j$,
in which the coefficients of $j$ are at least $-2$
(and at most $1$).
\end{proof}
%---------------------------------------------------------------------
Finally we have the main theorem.
%---------------------------------------------------------------------
\begin{theorem}
\label{thm:gin-of-4var-ci}
Let $R = \PPiv$, and $I = (x_1^a, x_2^b, x_3^c, x_4^d)$,
where at least one of $a$, $b$, $c$ and $d$ is equal to two.
Then the generic initial ideal of $I$ 
is equal to the almost revlex ideal 
corresponding to the same Hilbert function.
\end{theorem}
%---------------------------------------------------------------------
\begin{proof}
If one of $a$, $b$, $c$ and $d$ is equal to one,
the theorem is reduced to the case of three variables,
and follows from \cite{MR2371960} or \cite{MR2292574}.
We consider the case where $a, b, c, d \ge 2$.
First we show that $R/I$ has the 3-SLP
for $a, b, c \ge 2$ and $d = 2$.
Let $A = \PPiii/(x_1^a, x_2^b, x_3^c)$,
and $h_A$ the Hilbert function of $A$.
Then $A$ has the 3-SLP, and 
$h_A$ satisfies Conditions 
\textsf{(C1)}, \textsf{(C2)} and \textsf{(C3)},
since $h_A$ is unimodal and symmetric.
The difference $\Delta h_A$ is of the form
$(1, 2, \ldots, k, h'_k, h'_{k+1}, \ldots, h'_m)$,
where $k \ge h'_k \ge h'_{k+1} \ge \cdots \ge h'_m > 0$.
Hence $\Delta h_A$ satisfies Condition \textsf{(C1)}.
Condition \textsf{(C2)} for $\Delta h_A$ also holds,
since $h'_j - h'_{j+1} \le 2$ for $j = k-1, k, \ldots, m$
($h'_{k-1} = k$ and $h'_{m+1} = 0$)
by Lemma~\ref{lem:harima-sakaki}.
Therefore it follows from Lemma~\ref{lem:(53.2)} that
$R/I \simeq A \otimes_K K[x_4]/(x_4^2)$
has the 3-SLP.

Since the $k$-SLP is independent of the permutation of the variables,
$R/I$ has the 3-SLP,
if at least one of $a$, $b$, $c$ and $d$ is equal to two.
As noted in Definition~\ref{defn:k-slp},
the 2-SLP, 3-SLP and 4-SLP are equivalent for quotient rings of $R$,
it follows from Proposition~\ref{prop:gin-for-4-slp} 
that $\gin(I)$ is almost revlex.
\end{proof}
%---------------------------------------------------------------------
We conclude this note with the smallest example
which does not fit into our theorem.
%---------------------------------------------------------------------
\begin{remark}
\label{rem:x^3,y^3,z^3,w^3}
Let $R = \PPiv$ and $I = (x_1^3,x_2^3,x_3^3,x_4^3)$.
We can show that the quotient ring $R/I$ does not have the 2-SLP
as follows.

$R/I$ has the SLP, 
and we can fix a Lefschetz element $x_1+x_2+x_3+x_4$
by changing the coordinate if needed,
since Lefschetz elements of $R/I$ are of the form
$ax_1 + bx_2 + cx_3 + dx_4$ ($abcd \ne 0$).
Put 
\begin{gather*}
  A = \PPiii/(x_1^3,x_2^3,x_3^3,(x_1+x_2+x_3)^3)
  \simeq
  R/I+(x_1+x_2+x_3+x_4).
\end{gather*}
The Hilbert function of $A$ is $(1,3,6,6,3)$.
Let $\ell = ax_1+bx_2+cx_3$ be a general linear form of $A$, 
and we look at the rank of the linear mapping
$\times \ell^3:A_1 \to A_4$.
Using a computer we have
\begin{align*}
x_1 \ell^3 &= 
3a(b-c)^2 x_1^2x_3^2 +
3b(2a-c)(b-c) x_1x_2x_3^2 +
3b^2(a-c) x_2^2x_3^2,
\\
x_2 \ell^3 &= 
3a^2(b-c) x_1^2x_3^2 +
3a(a-c)(2b-c) x_1x_2x_3^2 +
3b(a-c)^2 x_2^2x_3^2,
\\
x_3 \ell^3 &= 
-3a^2(b-c) x_1^2x_3^2
-3ab(a+b-2c) x_1x_2x_3^2
-3b^2(a-c) x_2^2x_3^2,
\end{align*}
which are equations in $A$,
and $\{x_1^2x_3^2, x_1x_2x_3^2, x_2^2x_3^2\}$ is a basis of $A_4$.
We can show that the determinant of $\times\ell^3:A_1\to A_4$ 
is equal to zero independent of $a$, $b$ and $c$,
and therefore $A$ does not have the SLP.
Hence $R/I$ does not have the 2-SLP.
\end{remark}

%%%%%%%%%%%%%%%%%%%%%%%%%%%%%%%%%%%%%%%%%%%%%%%%%%%%%%%%%%%%%%%%%%%%%%
\bibliographystyle{alpha}
\bibliography{math}
%%%%%%%%%%%%%%%%%%%%%%%%%%%%%%%%%%%%%%%%%%%%%%%%%%%%%%%%%%%%%%%%%%%%%%
\end{document}